\documentclass[aos, preprint]{imsart}

\RequirePackage{amsthm,amsmath,amsfonts,amssymb}
\RequirePackage[authoryear]{natbib}%% uncomment this for author-year citations
\RequirePackage[colorlinks,citecolor=blue,urlcolor=blue]{hyperref}
\RequirePackage{graphicx}
\usepackage{subcaption}
\usepackage{lscape}
\usepackage{algorithmic}
\usepackage{algorithm}

\startlocaldefs
\theoremstyle{plain}

\newtheorem{theorem}{Theorem}[section]
\newtheorem{prop}[theorem]{Proposition}
\newtheorem{lemma}[theorem]{Lemma}
\theoremstyle{remark}

\newtheorem*{remark}{Remark}
\newcommand{\Ecal}{\mathcal{E}}

\newcommand{\Lcal}{\mathcal{L}}

\newcommand{\Ocal}{\mathcal{O}}

\newcommand{\Scal}{\mathcal{S}}

\newcommand{\Rb}{\mathbb{R}}

\newcommand{\Eb}{\mathbb{E}}
\newcommand{\Cb}{\mathbb{C}}
\newcommand{\Nb}{\mathbb{N}}

\newcommand{\BW}{\mathrm{BW}}

\newcommand{\what}{\widehat}
\newcommand{\T}{\top}

\newcommand{\Var}{\mathrm{Var}}
\newcommand{\tr}{\mathrm{tr}}
\newcommand{\vecop}{\mathrm{vec}}
\newcommand{\vech}{\mathrm{vech}}

\newcommand{\diag}{\mathrm{diag}}

\def\ba#1\ea{\begin{align*}#1\end{align*}} %\ba = \begin{algin*}, \ea = \end{align*}
\def\banum#1\eanum{\begin{align}#1\end{align}} %\banum = \begin{algin}, \eanum
\endlocaldefs

\begin{document}

\begin{frontmatter}
\title{Statistical properties of\\ matrix decomposition factor analysis}
%\title{A sample article title with some additional note\thanksref{t1}}
\runtitle{Statistical properties of matrix decomposition factor analysis}
%\thankstext{T1}{A sample additional note to the title.}

\begin{aug}
%%%%%%%%%%%%%%%%%%%%%%%%%%%%%%%%%%%%%%%%%%%%%%%
%% Only one address is permitted per author. %%
%% Only division, organization and e-mail is %%
%% included in the address.                  %%
%% Additional information can be included in %%
\author[A,B]{\fnms{Yoshikazu}~\snm{Terada}\ead[label=e1]{terada.yoshikazu.es@osaka-u.ac.jp}},
%\author[B]{\fnms{Second}~\snm{Author}\ead[label=e2]{second@somewhere.com}\orcid{0000-0000-0000-0000}}
%\and
%\author[B]{\fnms{Third}~\snm{Author}\ead[label=e3]{third@somewhere.com}}
%%%%%%%%%%%%%%%%%%%%%%%%%%%%%%%%%%%%%%%%%%%%%%
%% Addresses                                %%
%%%%%%%%%%%%%%%%%%%%%%%%%%%%%%%%%%%%%%%%%%%%%%
\address[A]{Graduate School of Engineering Science,
Osaka University\printead[presep={,\ }]{e1}}

\address[B]{Center for Advanced Integrated Intelligence Research,
RIKEN{}}
\end{aug}

\begin{abstract}
Numerous estimators have been proposed for factor analysis, and their statistical properties have been extensively studied.
In the early 2000s, a novel matrix factorization-based approach, known as Matrix Decomposition Factor Analysis (MDFA), 
was introduced and has been actively developed in computational statistics.
The MDFA estimator offers several advantages, including the guarantee of proper solutions (i.e., no Heywood cases) 
and the extensibility to $\ell_0$-sparse estimation. 
However, the MDFA estimator does not appear to be formulated as a classical M-estimator or a minimum discrepancy function (MDF) estimator, 
and the statistical properties of the MDFA estimator have remained largely unexplored.
Although the MDFA estimator minimizes a loss function resembling that of principal component analysis (PCA),
it empirically behaves more like consistent estimators used in factor analysis than like PCA itself.
This raises a fundamental question: can matrix decomposition factor analysis truly be regarded as ``factor analysis''?
To address this issue, we establish consistency and asymptotic normality of the MDFA estimator.
Recognizing that the MDFA estimator can be formulated as a semiparametric maximum likelihood estimator, 
we surprisingly find that the profile likelihood is given by the squared Bures-Wasserstein distance between the sample covariance matrix and the modeled covariance matrix.
As a consequence, the MDFA estimator is ultimately an MDF estimator for factor analysis.
Beyond MDFA, the same representation holds for a broad class of component analysis methods, including PCA, 
thereby offering a unified perspective on component analysis.
Numerical experiments demonstrate that MDFA performs competitively with other established estimators,
suggesting that it is a theoretically grounded and computationally appealing alternative for factor analysis.
\end{abstract}
%c
%\begin{keyword}[class=MSC]
%\kwd[Primary ]{00X00}
%\kwd{00X00}
%\kwd[; secondary ]{00X00}
%\end{keyword}


\end{frontmatter}
%%%%%%%%%%%%%%%%%%%%%%%%%%%%%%%%%%%%%%%%%%%%%%
%% Please use \tableofcontents for articles %%
%% with 50 pages and more                   %%
%%%%%%%%%%%%%%%%%%%%%%%%%%%%%%%%%%%%%%%%%%%%%%
%\tableofcontents

\section{Introduction}

Exploratory factor analysis, often referred to as factor analysis, is an important technique of multivariate analysis (\citealp{Anderson:2003}).
Factor analysis is a method for exploring the underlying structure of a set of variables and is applied in various fields.
In factor analysis, we consider the following model for a $p$-dimensional observation $x$:
\banum
x = \mu + \Lambda f + \epsilon, \label{eq:factor-model}
\eanum
where $\mu \in \Rb^p$ is a mean vector, $m$ is the number of factors ($m<p$), $\Lambda \in \Rb^{p\times m}$ is a factor loading matrix, 
$f$ be a $m$-dimensional centered random vector with the identity covariance,  
$\epsilon$ be a $p$-dimensional uncorrelated centered random vector, which is independent from $f$, 
with diagonal covariance matrix $\Var(\epsilon) = \Psi^2 = \diag(\sigma_1^2,\dots,\sigma_p^2)$.
Each component of $f$ and $\epsilon$ are called the common and unique factors, respectively.
The common factors refer to latent variables that influence observed variables, explaining their correlations.
In contrast, the unique factors represent all unique sources of variance in each observed variable independent of common factors.
For example, each unique factor includes a measurement error in each observed variable.
%p

For a constant $c_\Lambda>0$, 
let 
$\Theta_\Lambda:=\{ \Lambda \in \Rb^{p\times m} \mid |\lambda_{jk}| \le c_\Lambda \;(j=1,\dots,p;\;k=1,\dots,m)  \}$ 
be the parameter space for the factor loading matrix $\Lambda$.
For positive constants $c_L,c_U>0$, define the parameter space for $\Psi$ as 
$\Theta_\Psi:=\{ \diag(\sigma_1,\dots, \sigma_p) \mid c_L \le |\sigma_j| \le c_U\;(j=1,\dots,p) \}$.
Let $\Phi = [\Lambda, \Psi] \in \Rb^{p\times (m+p)}$, and define
$\Theta_\Phi = \{\Phi = [\Lambda, \Psi] \mid \Lambda \in \Theta_\Lambda\text{ and }\Psi \in \Theta_\Psi\}$.
For the factor model~(\ref{eq:factor-model}) with $\Phi=[\Lambda, \Psi]$, 
the covariance matrix of $x$ is represented as $\Phi\Phi^{\T} = \Lambda\Lambda^{\T} + \Psi^2$.

We assume that the factor model $(\ref{eq:factor-model})$ is true with 
some unknown parameter $\Phi_\ast = [\Lambda_\ast,\Psi_\ast] \in \Theta_\Phi$.
Let $\Sigma_\ast = \Phi_\ast \Phi_\ast^{\T} =\Lambda_\ast\Lambda_\ast^{\T} + \Psi_\ast^2$ denote the true covariance matrix.
It should be noted that the statistical properties described later still hold as a minimum contrast estimator
even when the factor model $(\ref{eq:factor-model})$ is not true.
Let $(x_1,f_1,\epsilon_1),\dots, (x_n,f_n,\epsilon_n)$ be i.i.d.~copies of $(x,f,\epsilon)$, 
where $(f_1,\epsilon_1),\dots,(f_n,\epsilon_n)$ are not observable in practice. 
Throughout the paper, it is assumed that $n > m+p$.
In factor analysis, we aim to estimate $(\Lambda_\ast,\Psi_\ast)$ 
from the observations $X_n=(x_1,\dots,x_n)^{\T}$.
Here, we note that the factor model $(\ref{eq:factor-model})$ has an indeterminacy.
For example, for any $m\times m$ orthogonal matrix $R$, 
a rotated loading matrix $\Lambda_\ast R$ can also serve as a true loading matrix.
Thus, let $\Theta_\Phi^\ast = \{\Phi \in \Theta_\Phi \mid \Sigma_\ast = \Phi\Phi^{\T}\}$ be 
the set of all possible true parameters.
%p

There are several estimation approaches to estimate the parameter $\Phi=[\Lambda,\Psi]$, e.g., 
maximum likelihood estimation, least-squares estimation, generalized least-squares estimation (\citealp{Joreskog:1972}), 
minimum rank factor analysis (\citealp{tenBergeKiers:1991}), and non-iterative estimation (\citealp{IharaKano:1986,Kano:1990}).
The theoretical properties of these estimation approaches have been extensively studied
(e.g., \citealp{Anderson:1956, Browne:1974,Kano:1983,Shapiro:1984,Kano:1986a, Anderson:1988,Kano:1991,ShapirotenBerge:2002}).
Moreover, most of these estimators can be formulated as minimum discrepancy function estimators.
Thus, we can apply the general theory of minimum discrepancy function estimators 
to derive the theoretical properties of the estimators (\citealp{Shapiro:1983, Shapiro:1984, Shapiro:1985a,Shapiro:1985b}).
%c

One might think that the maximum likelihood estimator is the best choice from the viewpoint of efficiency. 
However, it is well known that the maximum likelihood estimator, while robust to distributional assumptions, 
can perform poorly in the presence of certain types of model misspecification in (\ref{eq:factor-model}).
For more details, see \cite{MacCallumTucker91}, \cite{BriggsMacCallum03}, and \cite{MacCallumEtAl07}.
Thus, other estimators could be better choices than the maximum likelihood estimator in practice.
Some researchers recommend using the ordinary least squares estimator in exploratory factor analysis.
%c

In the early 2000s, a novel estimator based on matrix factorization was developed for factor analysis (\citealp{Socan:2003,deLeeuw:2004}).
According to \cite{Adachi:2018}, this method was originally developed by Professor Henk A.~L.~Kiers and first appeared in Socan's dissertation (\citealp{Socan:2003}).
This method is called matrix decomposition factor analysis (MDFA for short).
The MDFA algorithm always provides proper solutions (i.e., no Heywood cases in MDFA); 
thus, it is computationally more stable than the maximum likelihood estimator.
From the aspect of computational statistics, matrix decomposition factor analysis has been well-studied, 
and several extensions have been developed
(see, e.g., \citealp{Unkel:2010,Trendafilov:2011,Trendafilov:2013,Stegeman:2016,Adachi:2022,ChoHwang:23,Yamashita:24}).
An important extension for high-dimensional data is the sparse estimation of matrix decomposition factor analysis with the $\ell_0$-constraint.
Although the sparse estimation with the $\ell_0$-constraint is less biased than other sparse regularizations,
the optimization process with the $\ell_0$-constraint is generally challenging.
There is no sparse version of classical factor analysis with the $\ell_0$-constraint.
Surprisingly, the sparse MDFA estimator with the $\ell_0$-constraint can be easily obtained, 
as described in Section~\ref{sec:MDFA}.
%c

In matrix decomposition factor analysis, the estimator is obtained by minimizing the following principal component analysis-like loss function:
\[
\Lcal_n(\mu, \Lambda, \Psi, F,E)=\frac{1}{n}\sum_{i=1}^n \|x_i - (\mu + \Lambda f_i + \Psi e_i)   \|^2,
\]
where $e_i= (e_{i1},\dots,e_{ip})^{\T}$, $E = (e_1,\dots,e_n)^{\T} \in \Rb^{n\times p}$, and $F = (f_1,\dots,f_n)^{\T} \in \Rb^{n\times m}$.
As described in Section~\ref{sec:MDFA}, certain constraints are imposed on the common factor matrix $F$ and the normalized unique factor matrix $E$.
It is known that we cannot consider the maximum likelihood estimation to the problem of simultaneous estimation of $(\Lambda, \Psi)$ and latent factor scores $f_1,\dots,f_n$ (see Section 7.7 and Section 9 of \cite{Anderson:1956}).
However, interestingly, we consider the simultaneous estimation in matrix decomposition factor analysis, 
and the MDFA estimator can be interpreted as the maximum likelihood estimator of the semiparametric model.
More details on this point are provided in Section~\ref{sec:MDFA}.

It is important to note that the loss function of principal component analysis can be written as
\[
\Lcal_{\mathrm{PCA}}(\mu,\Lambda,F)=\frac{1}{n}\sum_{i=1}^n \|x_i - (\mu + \Lambda f_i)  \|^2,
\]
with the same constraints on $F$.
The equivalence between this formulation and other standard formulations of principal component analysis can be found in \cite{Adachi:2016}.

This formulation clearly shows that the term $\Psi e_i$ is the only difference between principal component analysis and matrix decomposition factor analysis.
Although the loss function of matrix decomposition factor analysis is very similar to that of principal component analysis,
the MDFA estimator empirically behaves like other consistent estimators used in factor analysis rather than principal component analysis.
In fact, \cite{Stegeman:2016} and \cite{Adachi:2018} empirically demonstrate that the matrix decomposition factor analysis provides very similar results 
with the classical consistent estimators, such as the maximum likelihood estimator.
For high-dimensional data, 
it is well-known that principal component analysis and factor analysis are approximately the same (e.g., \cite{Bentler:1990} and Section~2.1 of \cite{Fan:2013}). 
Thus, we can expect the MDFA estimator to exhibit similar behavior in high-dimensional settings. 
On the other hand, even in low-dimensional cases, matrix decomposition factor analysis yields results close to those of other consistent estimators 
for factor analysis.
%c

Unlike classical factor analysis, matrix decomposition factor analysis treats
the common factors $F$ and normalized unique factors $E$ as parameters that are estimated simultaneously 
with $\Phi=[\Lambda,\Psi]$.
The number of parameters linearly depends on the sample size $n$, 
and the standard asymptotic theory of classical M-estimators cannot be directly applied 
to analyze its theoretical properties.
As a result, the statistical properties of the MDFA estimator have yet to be discussed, 
leading to the open problem: Can matrix decomposition factor analysis truly be regarded as ``factor analysis''?
%Some researchers argue

In this paper, we establish the statistical properties of matrix decomposition factor analysis to answer this question.
We show that as the sample size $n$ goes to infinity, 
the MDFA estimator converges to the true parameter $\Phi_\ast \in \Theta_\Phi^\ast$.
First, we formulate the MDFA estimator as the semiparametric profile likelihood estimator 
and derive the explicit form of the profile likelihood.
Surprisingly, we discover that the profile likelihood is {\it the squared Bures-Wasserstein distance} between the sample covariance matrix and the modeled covariance matrix (e.g., \cite{BhatiaEtAl2019}).
%Next, we reveal the population-level loss function of matrix decomposition factor analysis and its fundamental properties.
Then, we show the statistical properties of matrix decomposition factor analysis.
All proofs are provided in the supplementary material.
%c

Throughout the paper, 
let us denote by $\lambda_j(A)$ the $j$th largest eigenvalue of a symmetric matrix $A$.
Let $\|\cdot\|_2$ and $\|\cdot\|_F$ represent the operator norm and the Frobenius norm for a matrix, respectively.
Let $I_p$ denote the identity matrix of size $p$, and let $O_{p\times q}$ denote the $p\times q$ matrix of zeros.
The $p$-dimensional vectors of ones will be denoted by $1_p$, and 
the vectors of zeros will be denoted by $0_p$.
%For matrix $A$, let $A^+$ denote the Moore-Penrose inverse of $A$.
We will denote by $\Ocal(p\times q)$ the set of all $p\times q$ column-orthogonal matrices 
and will denote by $\Ocal(p)$ the set of all $p\times p$ orthogonal matrices.
%c

\section{Matrix decomposition factor analysis (MDFA)}
\label{sec:MDFA}

We will briefly describe the matrix decomposition factor analysis proposed independently by \cite{Socan:2003} and \cite{deLeeuw:2004}.
Without loss of generality, the data matrix $X_n$ is centered by the sample mean, and we ignore the estimation of the mean vector $\mu$.
For simplicity of notation, we use the same symbol $X_n$ for the centered data matrix.
Let $F_n = (f_1,\dots,f_n)^{\T}$ and $\Ecal_n = (\epsilon_1,\dots,\epsilon_n)^{\T}$.
The factor model $(\ref{eq:factor-model})$ can be expressed in the matrix form:
$
X_n = F_n\Lambda^{\T} + \Ecal_n.
$
From this representation, 
we can naturally consider the following matrix factorization problem:
\[
X_n \approx F\Lambda^{\T} + E\Psi,
\]
where $F \in \Rb^{n\times m}$ and $E \in \Rb^{n\times p}$ are constrained by 
\banum
1_n^{\T}F =0_m^{\T},\; 1_n^{\T}E = 0_p^{\T},\;
\frac{1}{n}F^{\T}F = I_m,\; \frac{1}{n}E^{\T}E = I_p, 
\text{ and }
F^{\T}E = O_{m\times p}.
\label{eq:const}
\eanum
In the constraint~(\ref{eq:const}), 
the first two conditions are empirical counterparts of the assumptions $\Eb[f] = 0_{m}$ and $\Eb[\epsilon] = 0_p$.
Moreover, the remaining conditions correspond to the covariance constraints for $f$ and $\epsilon$.
This matrix factorization approach is called the matrix decomposition factor analysis (MDFA).
Here, we note that the factor $1/n$ can be replaced by the factor $1/(n-1)$ in the constraint~(\ref{eq:const}).
This modification is essential in practice, as will be described later.
%p

Let $\Theta_Z := \{Z=[F,E] \in \Rb^{n \times (m+p)} \mid Z \text{ satisfies the constraint } (\ref{eq:const})\}$.
In matrix decomposition factor analysis,
the estimator $(\what{\Lambda}_n,\what{\Psi}_n,\what{F}_n,\what{E}_n)$ is obtained by minimizing 
the following loss function over $\Phi=[\Lambda,\Psi] \in \Theta_\Phi$ and $Z=[F,E] \in \Theta_Z$:
\banum
\Lcal_n(\Phi, Z) 
%&= \frac{1}{n}\left\|X_n - Z\Phi^{\T} \right\|_F^2
= \frac{1}{n}\left\|X_n - (F\Lambda^{\T} + E\Psi) \right\|_F^2
=  \frac{1}{n}\sum_{i=1}^n\|x_i - (\Lambda f_i + \Psi e_i)\|^2.
\label{eq:empirical-loss}
\eanum
%c

Thus, we can formulate the matrix decomposition factor analysis as the maximum likelihood estimation of the following semiparametric model:
\banum
x_i = \Lambda f_i + \Psi e_i + \xi_i\quad(i = 1,\dots,n),
\label{eq:mdfa-mle}
\eanum
where both common and unique factor score vectors $f_1,\dots,f_n,e_1,\dots,e_n$ are fixed vectors satisfying the constraint~(\ref{eq:const}), and 
$\xi_1,\dots,\xi_n$ are independently distributed according to a $p$-dimensional centered normal distribution with a known variance $\tau_0^2I_p$.
Interestingly, as shown in Section~7.7 of \cite{Anderson:1956}, 
we cannot consider the maximum likelihood estimator when only the common factor score vectors are fixed.
Specifically, we cannot consider the maximum likelihood estimation of the following model:
\[
x_i = \Lambda f_i + \epsilon_i \quad(i = 1,\dots,n),
\]
where the common factor score vectors $f_1,\dots,f_n$ are fixed vectors satisfying the constraint~(\ref{eq:const}) for $F$, 
and $ \epsilon_1,\dots, \epsilon_n$ are independently distributed according to 
a $p$-dimensional centered normal distribution with unknown diagonal covariance $\Psi^2$.
%c

Now, we recall that the principal component analysis can be formulated as the minimization problem of the following loss function
with constraints $1_n^{\T}F = 0_m^{\T}$ and $F^{\T}F/n = I_m$:
\[
\frac{1}{n}\left\|X_n - F\Lambda^{\T} \right\|_F^2.
\]
The only difference between matrix decomposition factor analysis and principal component analysis is the term $E\Psi$ 
in the loss function (\ref{eq:empirical-loss}).
When we impose the constraint that $\Lambda^{\T}\Lambda$ is a diagonal matrix whose diagonal elements are arranged in decreasing order 
to identify the parameter $\Lambda$ uniquely,
the principal component estimator can be represented as
\[
\what{\Lambda}_\mathrm{PCA} = L_m\Delta_m/\sqrt{n}\;\text{ and }\;
\what{F}_\mathrm{PCA} = \sqrt{n}K_m,
\]
where $\Delta_m$ is a diagonal matrix with the first $m$ singular values of $X_n$, 
and $K_m \in \Ocal(n\times m)$ and $L_m \in \Ocal(p \times m)$ are 
the matrices of first $m$ left singular vectors and right singular vectors of $X_n$, respectively.
%c

\cite{Adachi:2018} shows several essential properties of matrix decomposition factor analysis.
Here, we will introduce some of these properties.
The loss function $\Lcal_n$ can be written as follows:
\banum
\Lcal_n(\Phi, Z) 
&= \frac{1}{n}\|X_n\|_F^2 + \|\Phi\|_F^2 - \frac{2}{n}\tr\{(X_n\Phi)^{\T} Z\} \label{eq:for-Z}\\
&= \frac{1}{n}\|X_n - ZZ^{\T} X_n/n\|_F^2 + \|X_n^{\T} Z/n - \Phi\|_F^2.\label{eq:for-Phi}
\eanum
%p

Now, we consider minimizing the loss function $\Lcal_n$.
Let $\what{K}(\Phi)\what{\Delta}(\Phi)\what{L}(\Phi)^{\T}$ be the singular value decomposition of $X_n \Phi/\sqrt{n}$,
where $\what{\Delta}(\Phi)$ is the diagonal matrix with the singular values, 
and $\what{K}(\Phi)\in \Ocal(n\times p)$ and $\what{L}(\Phi) \in \Ocal((m+p)\times p)$ 
are the matrix of the left singular vectors and the matrix of the right singular vectors, respectively.
Let $\what{S}_n = X_n^{\T} X_n/n$ be the sample covariance matrix,
and then the spectral decomposition of $\Phi^{\T} \what{S}_n \Phi$ can be written as
\[
\Phi^{\T} \what{S}_n \Phi = \left(X_n \Phi/\sqrt{n}\right)^{\T}\left(X_n \Phi/\sqrt{n}\right) 
= \what{L}(\Phi) \what{\Delta}(\Phi)^2 \what{L}(\Phi)^{\T}.
\]
From (\ref{eq:for-Z}), it follows that, for given $\Phi$, the following $\what{Z}(\Phi)$ attains the minimum of $\Lcal_n(\Phi,Z)$:
\banum
\what{Z}(\Phi) 
&= \sqrt{n}\what{K}(\Phi)\what{L}(\Phi)^{\T} + \sqrt{n}\what{K}_\perp(\Phi)\what{L}_\perp(\Phi)^{\T},
\label{eq:Z-min}
\eanum
where $\what{K}_\perp(\Phi) \in \Ocal(n\times m)$ and $\what{L}_\perp(\Phi) \in \Ocal((m+p)\times m)$ are 
column-orthonormal matrices such that $1_n^\top\what{K}_\perp(\Phi) = 0_m^\top$ and 
$
\what{K}(\Phi)^{\T}\what{K}_\perp(\Phi) =  \what{L}(\Phi)^{\T}\what{L}_\perp(\Phi) = O_{p\times m}.
$
It is important to note that $\what{K}_\perp(\Phi)$ and $\what{L}_\perp(\Phi)$ are not uniquely determined.
%p

For given $Z = [F,E] \in \Theta_Z$, the minimization of $\Lcal_n$ with $\Phi$ is obvious.
From (\ref{eq:for-Phi}), we conclude that, for given $Z \in \Theta_Z$, 
the optimal $\what{\Phi}(Z)$ is given by
\banum
\what{\Phi}(Z) = \Big[\what{\Lambda}(Z), \what{\Psi}(Z)\Big]  = \big[X_n^{\T} F/n, \diag(X_n^{\T} E/n)\big],
\label{eq:Phi-min}
\eanum
where $ \diag(X_n^{\T} E/n)$ is the diagonal matrix with diagonal elements of $X_n^{\T} E/n$.
That is, $\what{\Lambda}(Z) = X_n^{\T} F/n$ and $\what{\Psi}(Z) = \diag(X_n^{\T} E/n)$. 

Therefore, the minimization problem for $\Lcal_n$ can be solved by a simple alternating minimization algorithm.
An algorithm of matrix decomposition factor analysis is summarized in Algorithm~\ref{algo:MDFA}.
%\begin{itemize}
%\item[Step~1.] Initialize $t= 0$ and $\Phi_{(0)} = \Big[\Lambda_{(0)}, \Psi_{(0)} \Big]\in \Theta_\Phi$.
%\item[Step~2.] Update $t = t + 1$.
%		By the singular value decomposition of $X_n\Phi_{(t-1)}/\sqrt{n}$,
%		update the parameter $Z$ as follows:
%		\[
%		 \what{Z}_{(t)} = \Big[\what{F}_{(t)} , \what{E}_{(t)} \Big]
%		 = \sqrt{n}\what{K}(\Phi_{(t-1)})\what{L}(\Phi_{(t-1)})^{\T} + \sqrt{n}\what{K}_\perp(\Phi_{(t-1)})\what{L}_\perp(\Phi_{(t-1)})^{\T}.
%		\]
%\item[Step~3.]
%		Update the parameter $\Phi$ as follows:
%		\[
%		 \what{\Phi}_{(t)} = \Big[ \what{\Lambda}_{(t)},\what{\Psi}_{(t)}\Big]
%		  = \Big[ X_n^{\T}\what{F}_{(t)}/n, \diag\big(X_n^{\T}\what{E}_{(t)}/n\big)\Big].
%		\]
%\item[Step~4.]
%	Repeat Steps~2 and 3 until convergence.
%\end{itemize}
%c
\begin{algorithm}
\caption{An algorithm of matrix decomposition factor analysis.}
\label{algo:MDFA}
\begin{algorithmic}[1]
\STATE Initialize $t= 0$ and $\Phi_{(0)} = \Big[\Lambda_{(0)}, \Psi_{(0)} \Big]\in \Theta_\Phi$.
\STATE Update \( t = t + 1 \). By the singular value decomposition of $X_n\Phi_{(t-1)}/\sqrt{n}$,
		update the parameter $Z$ as follows:
		\[
		 \what{Z}_{(t)} = \Big[\what{F}_{(t)} , \what{E}_{(t)} \Big]
		 = \sqrt{n}\what{K}(\Phi_{(t-1)})\what{L}(\Phi_{(t-1)})^{\T} + \sqrt{n}\what{K}_\perp(\Phi_{(t-1)})\what{L}_\perp(\Phi_{(t-1)})^{\T}.
		\]\STATE  Update the parameter $\Phi$ as follows:
		\[
		 \what{\Phi}_{(t)} = \Big[ \what{\Lambda}_{(t)},\what{\Psi}_{(t)}\Big]
		  = \Big[ X_n^{\T}\what{F}_{(t)}/n, \diag\big(X_n^{\T}\what{E}_{(t)}/n\big)\Big].
		\]
\STATE Repeat Steps 2 and 3 until convergence.
\end{algorithmic}
\end{algorithm}

\begin{remark}\label{remark:algorithm}
For estimating both $\Phi$ and $Z$, the original data matrix $X_n$ is necessary.
However, when only the estimator for $\Phi$ is needed (i.e., the estimator of $Z$ is unnecessary),
\cite{Adachi:2012} shows that the algorithm can be performed using only the sample covariance matrix $\what{S}_n$.
Thus, for estimating only $\Phi$, the computational cost of matrix decomposition factor analysis only depends on the dimension $p$.
For more details of this algorithm, see \cite{Adachi:2018}.
\end{remark}

Due to the simplicity of the loss function, 
we can consider the sparse estimation of matrix decomposition factor analysis with the $\ell_0$-constraint $\|\Lambda\|_0\le C$,
where $\|\Lambda\|_0$ is the number of nonzero elements of $\Lambda$.
In Step 3, we keep the top $C$ largest elements in terms of absolute value and set others to zero.
This modified algorithm provides a solution for the minimization of the loss $\Lcal_n$ under the $\ell_0$-constraint $\|\Lambda\|_0\le C$.
This method is called the cardinality-constrained matrix decomposition factor analysis.
For more details, see \cite{Adachi:2016}.

\section{Asymptotic properties of matrix decomposition factor analysis}
\subsection{Main idea}\label{sec:idea}

First, we will describe the main idea for proving asymptotic properties of matrix decomposition factor analysis.
Remark~\ref{remark:algorithm} indicates the possibility that the loss function $\Lcal_n$ concentrated on $\Phi$ can be rewritten 
using the sample covariance matrix $\what{S}_n$ instead of the data matrix $X_n$.
The following lemma shows that the concentrated loss function $\Lcal_n(\Phi)=\min_{Z\in \Theta_Z} \Lcal_n(\Phi,Z)$ 
has the explicit form with the sample covariance matrix.
By considering $Z=[F,E]$ as the nuisance parameter in the semiparametric model (\ref{eq:mdfa-mle}), 
this loss $\Lcal_n(\Phi)$ is related to the concentrating-out (or profile likelihood) approach (\citealp{Newey:1994, Murphy:2000}).
That is,  this loss $\Lcal_n(\Phi)$ is the negative profile likelihood for the semiparametric model (\ref{eq:mdfa-mle}).
%p
\begin{prop}\label{lemma:marginal-loss}
For any $\Phi \in \Theta_\Phi$, 
\ba
\Lcal_n(\Phi)
&=\min_{Z\in \Theta_Z} \Lcal_n(\Phi,Z) = \tr(\what{S}_n)+ \tr(\Phi\Phi^\top) - 2\tr\left\{(\Phi^\top \what{S}_n\Phi)^{1/2}\right\}\\
&=\tr\bigl(\what{S}_n\bigr) + \tr\left\{ \Sigma(\Phi)\right\} -2\tr\left\{\left(\what{S}_n^{1/2}\Sigma(\Phi)\what{S}_n^{1/2}\right)^{1/2}\right\}
=d_{\BW}^2\left(\what{S}_n,\Sigma(\Phi)\right),
%&= \tr\left[ \{I_p-\what{A}(\Phi)\}^{\T} \what{S}_n\{I_p-\what{A}(\Phi)\}  \right]
% + \left\|(\Phi^{\T})^+\left(\Phi^{\T}\what{S}_n\Phi \right)^{1/2} - \Phi \right\|_F^2.
\ea
where $\Sigma(\Phi):=\Phi\Phi^T$ and $d_{\BW}(A,B)$ is the Bures-Wasserstein distance between positive semidefinite matrices $A$ and $B$.
\end{prop}
Suprisingly, the second line is known as {\it the squared Bures-Wasserstein distance between positive definite matrices}. 
Therefore, the MDFA estimator is finally formulated as a minimum discrepancy function estimator.
In \cite{Adachi:2012} and \cite{Adachi:2018}, 
the focus is primarily on the minimization algorithm and the descriptive properties of the minimizers.
Thus, the above representation of the concentrated loss is not obtained in the existing papers,
although their descriptive properties are essential to derive this representation.

\begin{remark}\label{remark:CA}
This result holds for any component analysis with the following loss function:
\[
\mathcal{L}_{\mathrm{CA}}(\Gamma, Z) = \|X - Z\Gamma^\top\|_F^2 \; \text{ subject to } \; Z^\top Z = I_q\;\text{ and }\; 1_n^\top Z = 0_q,
\]
where \( Z \) is an \( n \times q \) column-orthonormal matrix ($n > q$), and \( \Gamma \) is a \( p \times q \) parameter matrix.
That is, we have
\[
\mathcal{L}_{\mathrm{CA}}(\Gamma) = \min_Z \mathcal{L}_{\mathrm{CA}}(\Gamma, Z) = d_{\mathrm{BW}}^2\left(\widehat{S}_n, \Gamma\Gamma^\top\right).
\]
In this sense, the result offers a unified perspective on component analysis.
In particular, principal component analysis (PCA) can also be interpreted as the following MDF estimator using the squared Bures-Wasserstein distance:
\[
\mathcal{L}_{\mathrm{PCA}}(\Lambda) = \min_{F} \frac{1}{n} \|X - F\Lambda^\top\|_F^2 = d_{\mathrm{BW}}^2\left(\widehat{S}_n, \Lambda\Lambda^\top\right).
\]
A comparison of the following loss of MDFA with that of PCA highlights the differences between the two approaches more clearly:
\[
\Lcal_{\mathrm{MDFA}}(\Lambda,\Psi) :=\Lcal_n(\Lambda,\Psi) = d_{\mathrm{BW}}^2\left(\widehat{S}_n, \Lambda\Lambda^\top + \Psi^2\right).
\]
\end{remark}
%c
\begin{remark}\label{remark:unbias}
When we replace the factor $1/n$ with $1/(n-1)$ in the loss fucntion~(\ref{eq:empirical-loss}) and the constraint~(\ref{eq:const}), 
the sample covariance matrix $\what{S}_n$ is replaced by the unbiased covariance matrix $\what{U}_n=X^{\T}X/(n-1)$ in the explicit form of $\Lcal_n(\Phi)$.
This modification does not affect the asymptotic properties of the MDFA estimator but improves the finite-sample performance.
\end{remark}
%c

\subsection{Population-level loss function and its properties}
\label{sec:pop-lev}

%For given $\Phi\in \Rb^{p \times (p+m)}$, 
%the spectral decomposition of $\Phi^{\T}\Sigma_\ast \Phi \in \Rb^{(m+p)\times(m+p)}$ is denoted by
%\[
%\Phi^{\T}\Sigma_\ast \Phi = L(\Phi)\Delta(\Phi)^2L(\Phi)^{\T},
%\]
%where $\Delta(\Phi)^2$ is the diagonal matrix with ordered positive eigenvalues, and $L(\Phi) \in \Ocal((m+p)\times p)$ is the matrix of eigenvectors.
From Proposition~\ref{lemma:marginal-loss}, 
the population-level loss function can be defined as the squared Bures-Wasserstein distance between the true covariance and the modeled covariance:
\ba
\Lcal(\Phi)
= d_{\BW}^2\left(\Sigma_\ast,\Sigma(\Phi)\right)
=\tr\bigl(\Sigma_\ast \bigr) + \tr\left\{ \Sigma(\Phi)\right\} -2\tr\left\{\left(\Sigma_\ast^{1/2}\Sigma(\Phi)\Sigma_\ast^{1/2}\right)^{1/2}\right\}.
\ea
%Since the loss $\Lcal(\Phi)$ has a complex form, 
%it is not immediately clear whether this loss $\Lcal(\Phi)$ is reasonable for factor analysis.
The population-level matrix decomposition factor analysis is identifiable up to the indeterminacy of the factor model. 
\begin{prop}[Identifiability of the population-level MDFA]\label{prop:identifiability}
For any $\Phi = [\Lambda, \Psi] \in \Theta_\Phi$, the following two conditions are equivalent:
\[
\text{\rm (i)}\; \Sigma_\ast = \Phi\Phi^\T = \Lambda \Lambda^\T + \Psi^2,\;\text{ and }\;\text{\rm (ii)}\; \Lcal(\Phi) = 0.
\]
\end{prop}
\begin{proof}
The proof is immediately derived from the property of the  Bures-Wasserstein distance.
For example, see \cite{BhatiaEtAl2019}.
\end{proof}
%

%Interestingly, as shown in the proof, 
%only the second term of the loss $\Lcal$ is essential for this identifiability.
%Thus, the second term of the empirical loss $\Lcal_n$ could be an appropriate loss function for factor analysis,
%but the optimization step will be more complicated.
%p

%Since $\Lcal_n(\Phi)$ is not represented as the empirical mean over samples, %$\Lcal(\Phi)$ has a complex form.
%even the smoothness of the loss $\Lcal$ is non-trivial, unlike the classical theory of M-estimators.
The third term of the loss $\Lcal(\Phi)$ involves the square root of the non-degenerate matrix $\Sigma_\ast^{1/2}\Sigma(\Phi)\Sigma_\ast^{1/2}$. The differentiability of the matrix square root is guaranteed for non-degenerate matrices, and can also be extended to certain degenerate cases (e.g., see \cite{Freidlin:1968}).
Thus we can immediately ensure the smoothness of $\left(\Sigma_\ast^{1/2}\Sigma(\Phi)\Sigma_\ast^{1/2}\right)^{1/2}$.
The smoothness of the population-level loss $\Lcal(\Phi)$ is formally stated in the following proposition.
For $\Lambda \in \Rb^{p\times m}$ and $\Psi \in \Theta_\Psi$, 
let $\vecop(\Lambda)$ and $ \diag(\Psi)$ be the vectorization of $\Lambda$ and the diagonal vector of $\Psi$, respectively.
Write $\Theta_\phi =\big\{(\vecop(\Lambda)^\T, \diag(\Psi)^\T)^\T \in \Rb^{p(m+1)} \mid [\Lambda,\Psi] \in \Theta_\Phi\big\}$.

\begin{prop}[Smoothness of the population-level loss]\label{prop:loss-smoothness}
The population-level loss function $\Lcal:\Theta_\phi \rightarrow \Rb$ is smooth on the interior of the compact parameter space $\Theta_\phi$.
\end{prop}
\begin{proof}
The proof is immediate from Theorem~2 of \cite{Freidlin:1968}.
\end{proof}
%p

\subsection{Consistency}

Now, we will show the strong consistency of matrix decomposition factor analysis.
By the continuity of the squared Bures-Wasserstein distance between positive definite matrices, 
we can easily obtain the strong consistency of matrix decomposition factor analysis.

\begin{theorem}[Consistency of the MDFA estimator]\label{theorem:consistency}
Assume the observation $X_n=(x_1,\dots,x_n)^\T$ is an i.i.d. sample from the factor model (\ref{eq:factor-model}).
Let $\what{\Phi}_n = \Big[\what{\Lambda}_n, \what{\Psi}_n\Big]$ be 
the estimator of the matrix decomposition factor analysis for $\Phi=[\Lambda,\Psi]$.
That is,
$
\what{\Phi}_n \in \mathop{\arg\min}_{\Phi\in \Theta_\Phi} \Lcal_n(\Phi).
$
Then, 
\[
\lim_{n\rightarrow \infty}\Lcal\big(\what{\Phi}_n\big)= 0 \;\text{ a.s.},
\;\text{ and }\;
\lim_{n\rightarrow \infty}\min_{\Phi_\ast \in \Theta_\Phi^\ast}\big\| \what{\Phi}_n- \Phi_\ast \big\|_F = 0\;\text{ a.s.}
\]
\end{theorem}

From this theorem, the MDFA estimator converges to the true parameter, similar to other consistent estimators in factor analysis; 
thus, the MDFA estimator is appropriate for factor analysis.
This theorem explains why matrix decomposition factor analysis provides results 
similar to those of other consistent estimators in factor analysis.

Under Anderson and Rubin's sufficient condition,
a stronger consistency result can be achieved when combined with Theorem~1 of \cite{Kano:1983}.
More precisely, under Anderson and Rubin's sufficient condition, 
the factor decomposition $\Sigma_\ast = \Lambda_\ast\Lambda_\ast^{\T} + \Psi^2$ is unique, 
leading to the following result:
\[
 \lim_{n\rightarrow \infty}\min_{P\in \Ocal(m)}\big\| \what{\Lambda}_nP - \Lambda_\ast\big\|_F^2 = 0 \text{ and }
 \lim_{n\rightarrow \infty}\big\| \what{\Psi}_n - \Psi_\ast\big\|_F^2 = 0 \;\text{ a.s.}
\]

Moreover, even if the factor model (\ref{eq:factor-model}) is incorrect, 
the MDFA estimator remains consistent as a minimum contrast estimator. 
That is, the MDFA estimator converges to a solution that minimizes the population-level loss $\mathcal{L}$.
%p

\subsection{Asymptotic normality}

Based on the unified approach of \cite{Shapiro:1983}, 
we can show the asymptotic normality of the MDFA estimator.
To eliminate the rotational indeterminacy,
we consider the following identifiability condition introduced by \cite{Anderson:1956}:
\banum
\Lambda = 
\begin{bmatrix}
\lambda_{11} & 0 & 0 & \cdots & 0\\
\lambda_{21} & \lambda_{22} & 0 & \cdots & 0\\
\lambda_{31} & \lambda_{32} & \lambda_{33} & \cdots & 0\\
\vdots & \vdots & \vdots & \ddots & \vdots\\
\lambda_{m1} & \lambda_{m2} &  \lambda_{m3} & \cdots &  \lambda_{mm}\\
\vdots & \vdots & \vdots & \ddots & \vdots\\
\lambda_{p1} & \lambda_{p2} &  \lambda_{p3} & \cdots &  \lambda_{pm}
\end{bmatrix}.
\label{eq:IC5}
\eanum
When we impose this condition on the true loading matrix with $\lambda_{jj}> 0\;(j=1,\dots,m)$,
we can avoid both rotational and sign indeterminacy.
This condition can be easily handled in matrix decomposition factor analysis.
In fact, from the representation (\ref{eq:for-Phi}), 
we simply set the upper triangle part of $\what{\Lambda}_{(t)}$ to zero in Step~3 of Algorithm~\ref{algo:MDFA}.
This adjustment allows us to obtain the MDFA estimator that minimizes the loss function $\Lcal_n$ under the condition (\ref{eq:IC5}).

Moreover, for the identifiability of the factor decomposition, 
we also assume Anderson and Rubin's sufficient condition for the true loading matrix $\Lambda_\ast$. 
That is, we assume that if any row of $\Lambda_\ast$ is deleted, there remain two disjoint submatrices of rank $m$.

Let $\vech(\cdot)$ be the vech operator (see \cite{MagnusNeudecker:2019}).
For $\Phi=[\Lambda,\Psi]$ with the condition (\ref{eq:IC5}), 
let $\theta = (\lambda_{11},\dots,\lambda_{p1},\lambda_{22},\dots,\lambda_{p2},\dots,\lambda_{mm},\dots,\lambda_{pm},
\sigma_1^2,\dots,\sigma_p^2)^\T$ be the parameter vector.
%For some closed ball $B\subset \Rb^{pm - m(m-1)/2}$, 
Let $\Theta = \{\theta_\Lambda \in \Rb^{pm - m(m-1)/2} \mid |\theta_{\Lambda,j}| < c_\Lambda\} \times [c_L^2,c_U^2]^p$ be the compact parameter space for the parameter $\theta$.
To derive the asymptotic normality of the MDFA estimator,
we redefine the loss functions $\Lcal_n$ and $\Lcal$ as a unified function of the parameter $\theta$ and the covariance matrix $\Sigma$.
%For any $\Phi\in \Theta_\Phi$ and any $p\times p$ positive definite matrix $\Sigma>0$, 
%we rewrite the spectral decomposition of $\Phi^{\T}\Sigma \Phi \in \Rb^{(m+p)\times(m+p)}$ as
%\[
%\Phi^{\T}\Sigma\Phi = L(\Phi,\Sigma)\Delta(\Phi,\Sigma)^2L(\Phi,\Sigma)^{\T},
%\]
%where $\Delta(\Phi,\Sigma)^2$ is the diagonal matrix with ordered positive eigenvalues, 
%and $L(\Phi,\Sigma) \in \Ocal((m+p)\times p)$ is the matrix of eigenvectors.
For any $\theta \in \Theta$, define
\ba
\Lcal(\theta,\Sigma) := \Lcal(\Phi,\Sigma) = d_{\BW}^2(\Sigma(\theta),\Sigma) = d_{\BW}^2(\Sigma, \Sigma(\theta)),
\ea
where $\Phi$ is the matrix representation of $\theta$, and $\Sigma(\theta)$ is the modeled covariance matrix with $\theta$.
%, and $A(\Phi, \Sigma) = \Phi L(\Phi,\Sigma)L(\Phi,\Sigma)^{\T} \Phi^+$.
%Using this notation, we have $\Lcal_n(\Phi) = \Lcal(\Phi,\what{S}_n)$ and $\Lcal(\Phi) = \Lcal(\Phi,\Sigma_\ast)$.
%p

\begin{theorem}[Asymptotic normality of the MDFA estimator]\label{theorem:asymptotic-normality}
Assume that the observation $X_n=(x_1,\dots,x_n)^\T$ is an i.i.d. sample from the factor model (\ref{eq:factor-model}) with
the true parameter $\Phi_\ast = [\Lambda_\ast,\Psi_\ast]$.
It is assumed that the fourth moment of $x_1$ is bounded.
For the identifiability, suppose that the true loading matrix $\Lambda_\ast$ satisfies Anderson and Rubin's sufficient condition
and the condition~(\ref{eq:IC5}) with $\lambda_{jj}^\ast >0\;(j=1,\dots,m)$.
Let $\theta_\ast$ be the true parameter vector corresponding to $\Phi_\ast$.
Moreover, suppose that the true parameter lies in the interior of $\Theta$, and the Hessian matrix
\[
H_{\theta\theta}(\theta_\ast, \Sigma_\ast)=\frac{\partial^2\Lcal(\theta_\ast, \Sigma_\ast) }{\partial \theta \partial \theta^\T} 
\]
is nonsingular at the point $(\theta_\ast,\Sigma_\ast)$.

Let $\what{\theta}_n$ be the MDFA estimator with the condition~(\ref{eq:IC5}) for the true parameter $\theta_\ast$.
Here, the elements $\what{\lambda}_{jj}\;(j=1,\dots,m)$ of the estimator $\what{\theta}_n$ are set to be nonnegative.
Then, we obtain the asymptotic normality of the MDFA estimator $\what{\theta}_n$:
$
\sqrt{n}\big(\what{\theta}_n - \theta_\ast \big) \rightarrow N(0,V)
$
in distribution as $n\rightarrow \infty$, where the matrix $\Gamma$ is the asymptotic variance of the sample covariance, 
$V = J\Gamma J^\T$, and
\[
J =-H_{\theta\theta}^{-1}(\theta_\ast, \Sigma_\ast) \frac{\partial^2\Lcal(\theta_\ast, \Sigma_\ast)}{\partial \theta \partial \vech(\Sigma)^\T}.
\]
\end{theorem}

\section{Numerical experiments}
\begin{figure}[t]
  \begin{minipage}[b]{0.49\hsize}
    \centering
    \includegraphics[width = 0.98\textwidth]{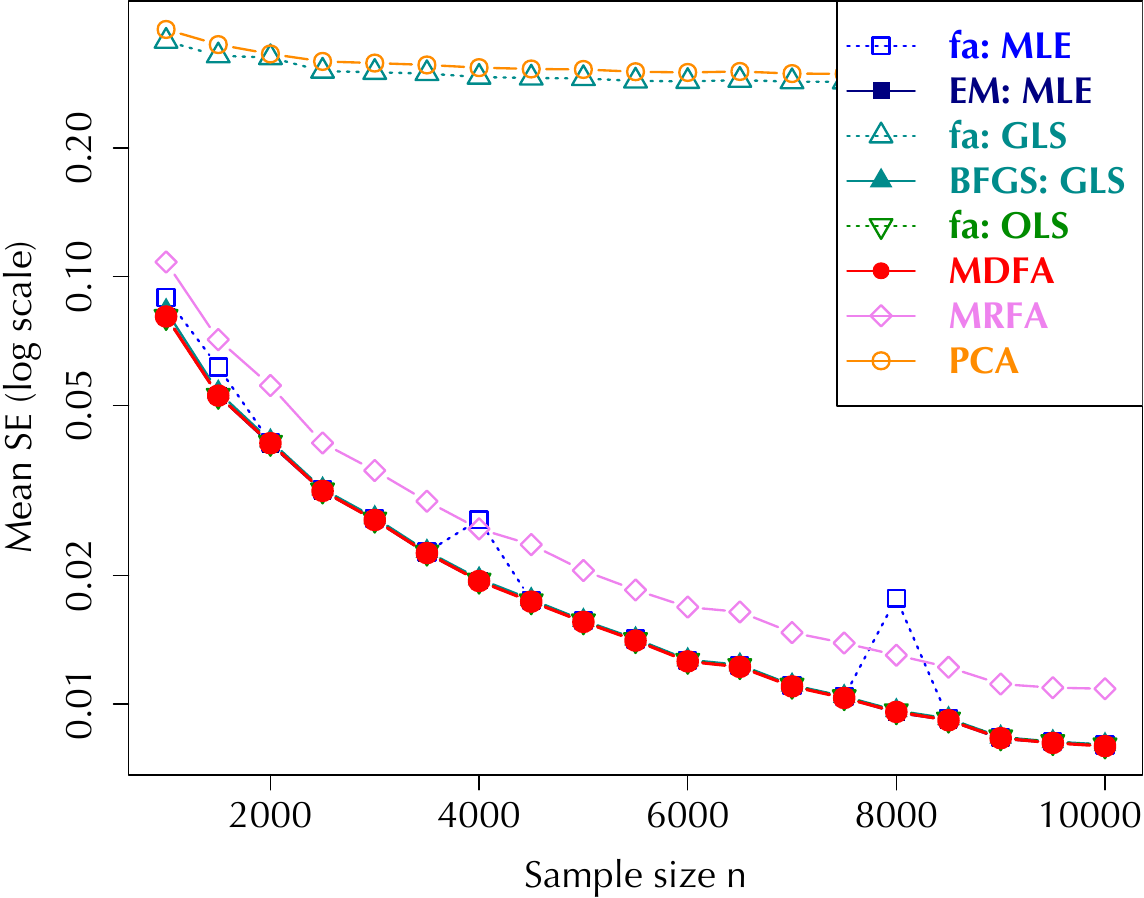}
    \subcaption{$\mathrm{SE}_{\Lambda}(\what{\Lambda}_n)$}\label{fig:setting1-lam}
  \end{minipage}
  \begin{minipage}[b]{0.49\hsize}
    \centering
    \includegraphics[width = 0.98\textwidth]{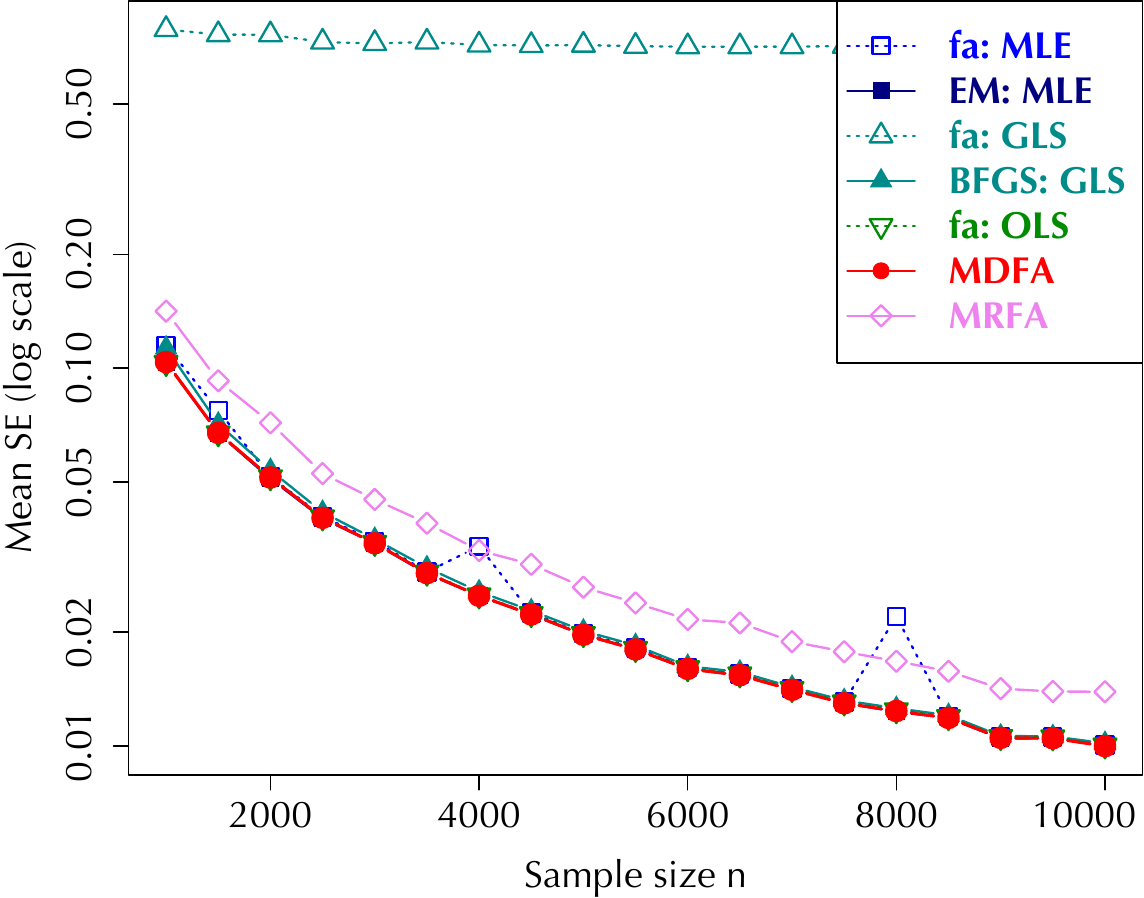}
    \subcaption{$\mathrm{SE}(\what{\Lambda}_n, \what{\Psi}_n^2)$}\label{fig:setting1-all}
  \end{minipage}
  \caption{Convergence behavior of all estimators in Setting~1 for larger sample sizes.}\label{fig:setting1}
\end{figure}

We will demonstrate the performance of the MDFA estimator through numerical experiments. 
Throughout the experiments, the MDFA estimator with the factor $1/(n-1)$ instead of $1/n$ is employed (see Remark~\ref{remark:unbias} in Section~\ref{sec:idea}).
We compare the MDFA estimator with existing estimators implemented in the R package \texttt{psych} (\citealp{psych}), 
specifically, the maximum likelihood estimator (fa: MLE), the ordinary least-squares estimator (fa: OLS), the generalized least-squares estimator (fa: GLS), and the minimum rank estimator (fa: MRFA).
In addition to these estimators, 
we also consider the maximum likelihood estimator by the expectation-maximization algorithm (EM: MLE) and 
the generalized least-squares estimator using numerical derivatives (BFGS: GLS)
since the maximum likelihood estimator and the generalized least-squares estimator of \texttt{psych} are sometimes unstable.
We note that the maximum likelihood estimator by the expectation-maximization algorithm never gives improper solutions (\citealp{Adachi:2013}).
Moreover, the principal component estimator (PCA) is also considered, although it is inconsistent as an estimator for the loading matrix in factor analysis.
%c

We consider the following four settings:
\begin{itemize}
\item {\bf Setting 1}: 
Set the number of variables $p = 20$ and the number of factors $m = 5$. 
The true loading matrix $\Lambda_\ast$ has a perfect simple structure.
That is, each row of the loading matrix has at most one non-zero element, and each column has four non-zero elements.
The non-zero elements of the loadings were independently drawn from uniform distributions on the interval $[0.90,0.95]$ for Factor 1, on $[0.85,0.90]$ for Factor 2, $[0.80,0.85]$ for Factor 3, 
$[0.45,0.50]$ for Factor 4, and $[0.40,0.45]$ for Factor 5.
Thus, the fourth and fifth factors are weak.
 For each variable, the unique variance was set to one minus the common variance.
 That is, the true covariance is given by
$\Sigma_\ast = \Lambda_\ast\Lambda_\ast^{\T} + \Psi_\ast^2$ and 
$\Psi_\ast^2 = I_p - \diag(\Lambda_\ast\Lambda_\ast^{\T})$.
 All $p$-dimensional observations are independently generated from the centered multivariate normal distribution with the true covariance $\Sigma_\ast$.
 
\item {\bf Setting 2}: 
Set $p = 20$ and $m = 5$.
The generating process of the true loading matrix is the same as in Setting 1.
In this setting, we consider the following approximate factor model:
\[
\Sigma_\ast = \Lambda_\ast\Lambda_\ast^{\T} + \Psi_\ast^2 + WW^{\T},
\]
where $W$ is a $p\times 150$-minor factor loading matrix, and $WW^{\T}$ represents the model error term.
Figure~\ref{fig:setting-2} shows an example of the true covariance matrix $\Sigma_\ast$.
Minor factor loadings in $W$ were generated using the well-known procedure of \cite{TuckerEtAl:1969}. 
This procedure has two parameters: the proportion $\pi$ of unique variance due to model error and the parameter $\varepsilon$ introduced by \cite{TuckerEtAl:1969}. 
Here, we use the setting of \cite{BriggsMacCallum03} for generating $W$. 
The proportion of unique variance was set to $\pi= 0.2$. 
The parameter $\varepsilon$ controls the size of the factor loadings in successive minor factors, and we set $\varepsilon$ to $0.1$.
The unique variances were set such that the diagonal elements of $\Sigma_\ast$ equal $1.0$.
In this experiment, we use the function \texttt{simFA} in the R package \texttt{fungible} (\citealp{fungible}) for generating $W$.

\item {\bf Setting 3}:  Set $p = 50$ and $m = 5$, with each factor loading having ten non-zero elements. 
All other aspects are identical to Setting 1.

\item {\bf Setting 4}:  Set $p = 50$ and $m = 5$, with each factor loading having ten non-zero elements. 
All other aspects are identical to Setting 2.
\end{itemize}
%c

\begin{figure}
    \centering
    \includegraphics[width=0.98\textwidth]{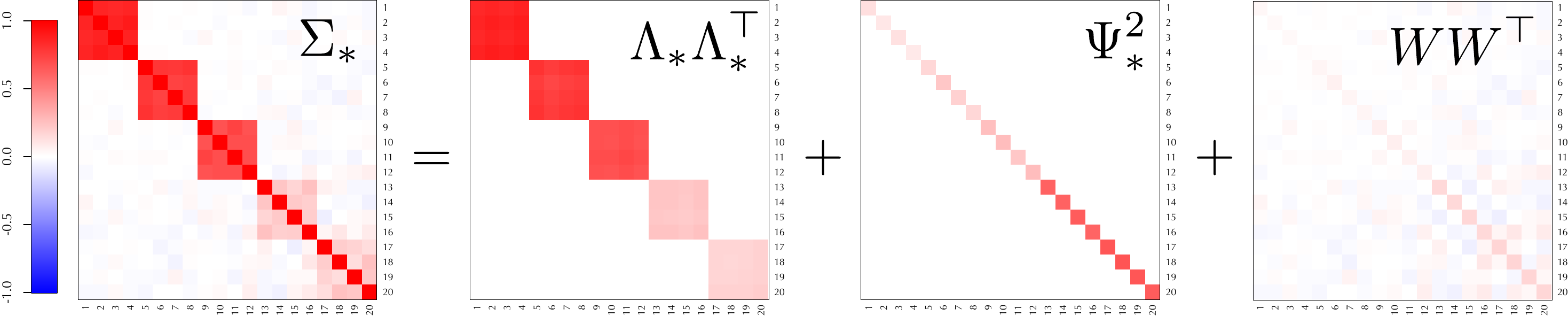}
  \caption{An example of the true covariance $\Sigma_\ast=\Lambda_\ast\Lambda_\ast^{\T} + \Psi_\ast^2 + WW^{\T}$ in Setting~2.}\label{fig:setting-2}
\end{figure}

\begin{remark}
Settings 2 and 4 are inspired by the numerical experiments of \cite{BriggsMacCallum03}.
In these settings, the ordinary least-squares estimator is known to perform better than the maximum likelihood estimator.
Specifically, the maximum likelihood estimator often fails to recover the weak factors (i.e., the fourth and fifth factors), whereas the ordinary least-squares estimator succeeds.
\end{remark}

In each setting, ten different sample sizes $n = 100, 200, \dots,1000$ are considered, 
and 100 replications were generated at each sample size.
To demonstrate the consistency of the MDFA estimator more clearly, 
we also consider larger sample sizes $n = 1000,2000,\dots,10000$ for Setting~1.
To evaluate the estimation accuracy, we employ the following squared errors: 
\[
\mathrm{SE}_{\Lambda}(\what{\Lambda}_n) = \min_{P: P^{\T}P = I_m}\big\| \what{\Lambda}_nP - \Lambda_\ast\big\|_F^2,
\;\text{ and }\;
\mathrm{SE}(\what{\Lambda}_n, \what{\Psi}_n^2) = \mathrm{SE}_{\Lambda}(\what{\Lambda}_n)  + \big\|\what{\Psi}_n^2 - \Psi_\ast^2\big\|_F^2,
\]
where $\what{\Lambda}_n$ and $\what{\Psi}_n^2$ are estimators for $\what{\Lambda}_n$ and $\Psi_\ast^2$, respectively.
For Settings~2 and 4, the true unique variance $\Psi_\ast^2$ is replaced by $I_p-\diag(\Lambda_\ast\Lambda_\ast^{\T})$ 
when we evaluate the estimation accuracy $\mathrm{SE}(\what{\Lambda}_n,\what{\Psi}_n^2)$ for unique variance.
For the principal component estimator, we only evaluate the squared error $\mathrm{SE}_{\Lambda}(\what{\Lambda}_n)$ for factor loadings.
All reported results are the average of 100 replications. 

Figure~\ref{fig:setting1} shows the convergence behavior of all estimators in Setting~1 with larger sample sizes.
The MDFA estimator converges to the true parameter as the other consistent estimators.
In contrast, the principal component estimator is inconsistent in the finite-dimensional settings, 
and thus, it converges to a different solution.
Although the generalized least-squares estimator using numerical derivatives performs well, 
the generalized least-squares estimator by \texttt{psych} does not work.
This difference is related to the choice of optimization methods.

\begin{figure}
    \centering
    \includegraphics[width=0.95\textwidth]{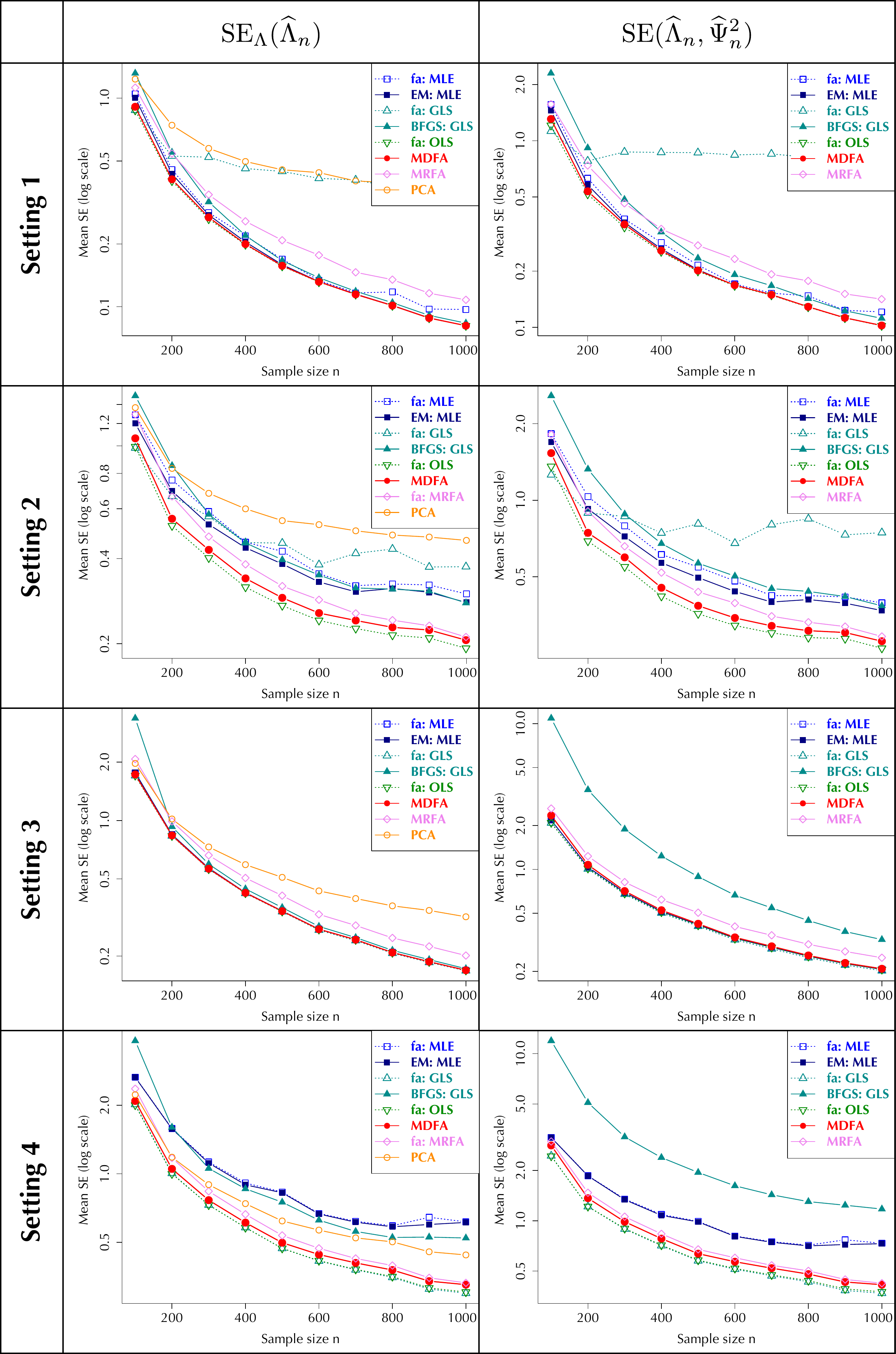}
  \caption{The averages of squared errors of $100$ replications for all settings.}\label{fig:all}
\end{figure}

Figure~\ref{fig:all} shows the result of all settings with sizes $n = 100,\dots, 1000$.
In Settings 2 and 4, the maximum likelihood estimator fails to recover the weak factor loadings (Factors 4 and 5), 
although it performs well in the settings without model error.
In contrast, the minimum rank estimator works well when model error is present
but yields slightly poorer results than other consistent estimators in settings without model error.
Due to the complexity of the optimization in minimum rank factor analysis, 
its computational cost is higher than that of other methods.
Unexpectedly, the generalized least-squares estimators are unstable.
In higher-dimensional settings, the generalized least-squares estimator by \texttt{psych} performs well,
although it behaves irregularly in the lower-dimensional settings.
The ordinary least-squares estimator provides better performance across all settings, 
as shown in \cite{MacCallumTucker91}, \cite{BriggsMacCallum03}, and \cite{MacCallumEtAl07}.
Remarkably, the MDFA estimator also performs consistently well in all settings, 
comparable to the ordinary least-squares estimator.

\section{Conclusion}

In this paper, we established the statistical properties of matrix decomposition factor analysis for the first time. 
We proved that the MDFA estimator is consistent as a factor analysis method, rather than a variant of principal component analysis (PCA).
Furthermore, we established the asymptotic normality of the MDFA estimator based on the unified theory of \cite{Shapiro:1983}.
A key finding is that, by recognizing the MDFA estimator as a semiparametric maximum likelihood estimator,
the profile likelihood is found to be the squared Bures-Wasserstein distance between the sample covariance matrix and the modeled covariance matrix.
Beyond MDFA, the squared Bures-Wasserstein representation holds for a broad class of component analysis methods, including PCA, 
thereby offering a unified perspective on component analysis.

Numerical experiments demonstrated that the MDFA estimator performs stably and effectively across various settings.
These results indicate that the MDFA estimator is a promising choice for factor analysis, 
and that it can potentially be extended to structural equation modeling (SEM).
However, because oblique models are essential in SEM, extending MDFA to SEM becomes computationally challenging.

As described in Section~\ref{sec:MDFA}, 
the sparse MDFA estimator with the $\ell_0$-constraint can be easily obtained.
An important direction for future research is to investigate the asymptotic properties of the sparse MDFA estimator in high-dimensional settings.

%\newpage

\begin{appendix}
\section{Proofs}
\subsection{Proof of Proposition~\ref{lemma:marginal-loss}}

\begin{proof}[Proof of Proposition~\ref{lemma:marginal-loss}]
From Equation~(7) in the main paper, we have
\ba
\Lcal_n(\Phi) = \Lcal_n\big(\Phi, \what{Z}(\Phi)\big) 
=
\frac{1}{n}\|X_n\|_F^2 + \|\Phi\|_F^2 - \frac{2}{n}\tr\{(X_n\Phi)^{\T} \what{Z}(\Phi)\}.
\ea
By the definition of $\what{Z}(\Phi)$, 
\ba
(X_n\Phi)^{\T} \what{Z}(\Phi)
&=
\what{L}(\Phi)\what{\Delta}(\Phi)\what{K}(\Phi)^\top\left( \what{K}(\Phi)\what{L}(\Phi)^\top + \what{K}_\perp(\Phi)\what{L}_\perp(\Phi)^\top \right)\\
&=
\what{L}(\Phi)\what{\Delta}(\Phi)\what{L}(\Phi)^\top
=
\left(\Phi^\top \what{S}_n \Phi \right)^{1/2}.
\ea
Thus, we obtain the following representation:
\ba
\Lcal_n(\Phi) =
\tr\left(\what{S}_n\right) + \tr(\Phi\Phi^\top)  - 2 \tr\left\{ \left(\Phi^\top \what{S}_n \Phi \right)^{1/2}\right\}.
\ea
By the property of the eigenvalues, 
the eigenvalues of $\Phi^\top \what{S}_n \Phi$ are identical to $\what{S}_n^{1/2} \Phi\Phi^\top \what{S}_n^{1/2}$.
From this, we finally obtain 
\ba
\Lcal_n(\Phi) =
\tr\left(\what{S}_n\right) + \tr\left\{ \Sigma(\Phi)\right\}  - 2 \tr\left\{ \left(\what{S}_n^{1/2} \Sigma(\Phi) \what{S}_n^{1/2} \right)^{1/2} \right\}.
\ea
\end{proof}

\subsection{Proof of the main theorems}

\begin{proof}[Proof of Theorem~3.5]
By the triangle inequality of $d_{\BW}$, we have
\[
\Lcal_n(\Phi) = d_{\BW}^2\big(\what{S}_n,\Sigma(\Phi)\big)
\le \left\{ d_{\BW}\big(\what{S}_n,\Sigma_\ast \big) + d_{\BW}\left(\Sigma_\ast,\Sigma(\Phi)\right) \right\}^2
= \left\{ d_{\BW}\big(\what{S}_n,\Sigma_\ast \big) +  \Lcal^{1/2}(\Phi)  \right\}^2.
\]
The smoothness of $d_{\BW}$ implies the first claim immediately:
\ba
\Lcal\big(\what{\Phi}_n\big)   -  \Lcal(\Phi_\ast)
&\le
\Lcal\big(\what{\Phi}_n\big)  - \Lcal_n\big(\what{\Phi}_n\big) + \Lcal_n\big(\what{\Phi}_n\big)  -  \Lcal(\Phi_\ast)\\
&\le
\Lcal\big(\what{\Phi}_n\big)  - \Lcal_n\big(\what{\Phi}_n\big) + \Lcal_n(\Phi_\ast) -  \Lcal(\Phi_\ast)\\
&\le 
2\sup_{\Phi \in \Theta_\Phi}\left|\Lcal_n(\Phi) -  \Lcal(\Phi) \right| \rightarrow 0 \quad\text{a.s.} 
\ea

Next, we will show the second claim.
Let $d(\Phi,\Theta_\Phi^\ast) = \min_{\Phi_\ast \in \Theta_\Phi^\ast}\big\| \Phi- \Phi_\ast \big\|_F$.
The continuity of $\Lcal$ implies
\banum
\forall \epsilon >0;\; \inf_{d(\Phi,\Theta_\Phi^\ast)\ge \epsilon} \Lcal(\Phi) > 0. \label{eq:identifiability}
\eanum
By the above uniform law of large numbers, we have
\ba
\mathop{\lim\inf}_{n\rightarrow \infty}\inf_{d(\Phi,\Theta_\Phi^\ast)\ge \epsilon} \Lcal_n(\Phi)
\ge
\inf_{d(\Phi,\Theta_\Phi^\ast)\ge \epsilon} \Lcal(\Phi) \quad\text{a.s.}
\ea
Let $\tilde{\Omega}$ be the event satisfying the above uniform convergence.
For any $\epsilon>0$ and $\omega \in \tilde{\Omega}$, 
the inequality (\ref{eq:identifiability}) and the first claim lead to
\ba
\mathop{\lim\inf}_{n\rightarrow \infty}\inf_{d(\Phi,\Theta_\Phi^\ast)\ge \epsilon} \Lcal_n(\Phi)
\ge
\inf_{d(\Phi,\Theta_\Phi^\ast)\ge \epsilon} \Lcal(\Phi)
>
\Lcal(\Phi^\ast) = \mathop{\lim\sup}_{n\rightarrow \infty} \Lcal_n\big(\what{\Phi}_n\big), 
\ea
and thus there exists $n_0 \in \Nb$ such that for all $n\ge n_0$
\ba
\inf_{d(\Phi,\Theta_\Phi^\ast)\ge \epsilon} \Lcal_n(\Phi) >  \Lcal_n\big(\what{\Phi}_n\big).
\ea
This implies $d\big(\what{\Phi}_n,\Theta_\Phi^\ast\big)\le \epsilon$, and the proof is complete.
\end{proof}

\begin{proof}[Proof of Theorem~\ref{theorem:asymptotic-normality}]
We employ the unified approach of \cite{Shapiro:1983} to show the asymptotic normality of the MDFA estimator.
By Lemma~\ref{lemma:integral-form}, we can ensure the smoothness of $\Lcal(\theta,\Sigma)$ on $\Theta^\circ \times \Scal_{++}^p$,
where $\Theta^\circ$ is the interior of $\Theta$, and $\Scal_{++}^p$ is the set of all $p\times p$ positive definite matices.
By the assumptions and Theorem~4.1 of \cite{Shapiro:1983}, 
for any $\Sigma$ in a sufficiently small neighbourhood of $\Sigma_\ast$,
the minimizer of $\Lcal(\theta,\Sigma)$ is unique.
Thus, for $\Sigma$ in the neighbourhood of $\Sigma_\ast$, let $\what{\theta}(\vech(\Sigma))$ be the minimizer of $\Lcal(\theta,\Sigma)$.
By the optimality of $\what{\theta}(\vech(\Sigma))$ and the implicit function theorem,
the MDFA estimator $\what{\theta}(\vech(\Sigma))$  is continuously differentiable at $\Sigma = \Sigma_\ast$, and
the Jacobian matrix of $\what{\theta}(\vech(\Sigma))$ at $\Sigma = \Sigma_\ast$ is
\[
J=\frac{\partial \what{\theta}(\vech(\Sigma)) }{\partial \vech(\Sigma)} 
= -H_{\theta\theta}^{-1}(\theta_0,\Sigma_\ast)\frac{\partial^2\Lcal(\theta_\ast, \Sigma_\ast)}{\partial \theta \partial \vech(\Sigma)^{T}}.
\]
Thus, we have
\ba
&\what{\theta}(\vech(\Sigma)) -  \theta_\ast
= \what{\theta}(\vech(\Sigma)) - \what{\theta}(\vech(\Sigma_\ast))\\
&= \{\vech(\Sigma) -  \vech(\Sigma_0)\}^{T}
\int_{0}^1 \frac{\partial}{\partial \vech(\Sigma)} \what{\theta}\big(\vech(\Sigma_\ast) + t\{\vech(\Sigma) - \vech(\Sigma_\ast)\} \big)\,dt.
\ea
By the asymptotic normality of the sample covariance and Slutsky's theorem,
we can conclude that
\ba
\sqrt{n} \left( \what{\theta}_n -  \theta_\ast \right)
&=
\sqrt{n} \left\{ \what{\theta}(\what{S}_n) -  \theta_\ast \right\}
=
\frac{\partial \what{\theta}(\vech(\Sigma))}{\partial \vech(\Sigma)^{T}} \sqrt{n} \{\vech(\Sigma) -  \vech(\Sigma_0)\}
+o_P(1)\\
&\rightarrow N(0,V)
\ea
in distribution as $n\rightarrow \infty$.
\end{proof}

%\newpage
\subsection{Some auxiliary lemmas}
%XXX Move to the supplementary material?XXX

Here, we describe some auxiliary results.

\begin{lemma}[Some basic properties]\label{lemma:basic}
For any $\Phi\in \Theta_\Phi$, the following properties hold:
\ba
\text{\rm (a) }
&c_L^2 \le \lambda_j(\Phi\Phi^{\T})\;(j=1,\dots,p),\;\;
\text{\rm (b) }
\|\Phi\|_F^2 =\|\Lambda\|_F^2 + \sum_{j=1}^p\sigma_j^2 \le p(c_\Lambda^2  + c_U^2), \;\text{and}\\
\text{\rm (c) }
&\lambda_p(\Phi^{\T}\Sigma_\ast \Phi) = \|\Sigma_\ast^{1/2} \Phi\|_2^2 \ge c_L^4.
%\text{\rm (e) }
%&\|\what{A}(\Phi)\|_F
%\le
%\|\Phi\|_F\|\what{L}(\Phi)\what{L}(\Phi)^{\T}\|_2\|\Phi^{+} \|_F
%\le
%p\frac{(c_\Lambda^2  + c_U^2)^{1/2}}{c_L}.
\ea
\end{lemma}
\begin{proof}
(a)
For all $x \in \Rb^p$,
\[
x^{\T} (\Phi\Phi^{\T}) x = x^{\T} \Lambda\Lambda^{\T} x + x^{\T} \Psi^2 x
\ge x^{\T} \Psi^2 x = \sum_{j = 1}^p \sigma_j^2x_j^2\ge c_L^2\|x\|_2^2.
\]

\noindent(b) By the definition of $\Phi$,
$
\|\Phi\|_F^2 =\|\Lambda\|_F^2 + \sum_{j=1}^p\sigma_j^2 \le p(c_\Lambda^2  + c_U^2).
$

\noindent(c)
For $x \in \Rb^{m+p}$ with $\|x\|=1$ and $\Phi x \neq 0$, we have
\[
\|\Sigma_\ast^{1/2}\Phi x\| = \left\|\Sigma_\ast^{1/2} \frac{\Phi x}{\|\Phi x\|} \right\|\|\Phi x\|
\ge 
\inf_{\|y\| = 1}  \left\|\Sigma_\ast^{1/2}y\right\|  \inf_{\|x\|=1,\Phi x\neq 0}\|\Phi x\|
\ge c_L^2 > 0.
\]
Thus, we obtain
$
\lambda_p(\Phi^{\T}\Sigma_\ast \Phi) = \|\Sigma_\ast^{1/2} \Phi\|_2^2 \ge c_L^4.
$
\end{proof}

To see the smoothness of $d_{\BW}$, we present a slight extension of Theorem~2 of \cite{Freidlin:1968} for a general matrix function.

\begin{lemma}[Integral form and smoothness of a matrix function]\label{lemma:integral-form}
Let $D\subset \Rb^d$ be the bounded region, 
and consider a symmetric and smooth matrix-valued function $A: \Rb^d \rightarrow \Rb^{p\times p}$.
Suppose that $A(x)$ is non-negative and has a constant rank (say $q$) everywhere in the closure $\bar{D}$ of the domain $D$.
Let $f$ be a function $f:\Cb \rightarrow \Cb$, and suppose that $f$ is analytic except at the origin.
We define 
\[
F(x) = \sum_{j=1}^q f(\lambda_j(x))e_j(x)e_j(x)^{\T},
\]
where $\lambda_j(x)$ and $e_j(x)$ are the $j$th positive eigenvalue and the corresponding eigenvector of $A(x)$, respectively.

Then, the matrix function $F(x)$ has the following integral form
\[
F(x) = \frac{1}{2\pi i} \int_C f(z) \{zI_p - A(x)\}^{-1}\,dz,
\]
where $C$ is a closed loop that lies in the open right half-plane and contains all positive eigenvalues of $A(x)$ inside itself.
Thus, $F(x)$ has the same smoothness as $A(x)$ on $D$.
\end{lemma}
\begin{proof}
The proof is almost the same as the proof of Theorem~2.2 of Chapter III in \cite{Freidlin:1985}.
\end{proof}

\end{appendix}

\begin{funding}
The author was supported by JSPS KAKENHI Grant (JP20K19756, JP20H00601, and JP23H03355) 
and  iGCORE Open Collaborative Research.
%The first author was supported by NSF Grant DMS-??-??????.
%
%The second author was supported in part by NIH Grant ???????????.
\end{funding}

\bibliographystyle{imsart-nameyear}
\bibliography{paper-ref}

\end{document}